\documentclass[11pt,twoside,]{amsart}
\usepackage[dvips]{graphicx}
\usepackage{tikz,tkz-berge,tkz-graph}
\usepackage{longtable}
\usepackage{amsmath, amsthm, amscd, amsfonts, amssymb, color, comment,verbatim}
\usepackage[bookmarksnumbered, plainpages]{hyperref}
 \usepackage[usenames,dvipsnames]{pstricks}
 \usepackage{epsfig}
 \usepackage{pst-grad} % For gradients
 \usepackage{pst-plot} % For axes
\newtheorem{thm}{Theorem}[section]
\newtheorem{cor}[thm]{Corollary}
\newtheorem{lem}[thm]{Lemma}
\newtheorem{prop}[thm]{Proposition}
\newtheorem{exm}[thm]{Example}
\newtheorem{defn}[thm]{Definition}

\newtheorem{rem}[thm]{Remark}

\numberwithin{equation}{section}
\begin{document}

\oddsidemargin 0mm
\evensidemargin 0mm

\thispagestyle{plain}

\vspace{5cc}
\begin{center}

% Title and Names ----------------------------------------------------------------------------------------------------------

{\large\bf Zero divisors of support size $3$ in group algebras and trinomials divided by irreducible polynomials over $GF(2)$}
\rule{0mm}{6mm}\renewcommand{\thefootnote}{}
\footnotetext{
{\rule{2.4mm}{0mm}Keywords and Phrases. Group algebra, torsion-free, zero divisor, support size, trinomial.}}

\vspace{1cc}
{\large\it Alireza Abdollahi and Zahra Taheri}

% Abstract --------------------------------------------------------------------------------------------------------------------

\vspace{1cc}
\parbox{27cc}{{\small

\textbf{Abstract.} 
 A famous conjecture about group algebras of torsion-free groups states that there is no zero divisor in such group algebras. A recent approach to settle the conjecture is to show the non-existence of zero divisors with respect to the length of possible ones, where by the length we mean the size of the support of an element of the group algebra. The case length $2$ cannot be happen. The first unsettled case is the existence of zero divisors of length $3$. 
 Here we study possible length $3$ zero divisors in the rational group algebras and in the group algebras over the field $\mathbb{F}_p$ with $p$ elements for some prime $p$.
 As a consequence we prove that the rational group algebras of torsion-free groups which are residually finite $p$-group for some prime $p\neq 3$ have no zero divisor of length $3$.   We note that the determination of all zero divisors of length $3$ in  group algebras over $\mathbb{F}_2$ of cyclic groups is equivalent to find all trinomials (polynomials with 3 non-zero terms) divided by irreducible polynomials over $\mathbb{F}_2$. The latter is a subject  studied in coding theory and we add here some results, e.g. we show that $1+x+x^2$ is a zero divisor in the group algebra  over $\mathbb{F}_2$ for some element $x$ of the group  if and only if $x$ is of finite order divided by $3$ and we find all $\beta$ in the group algebra of the shortest length such that $(1+x+x^2)\beta=0$; and $1+x^2+x^3$ or $1+x+x^3$ is a zero divisor in the group algebra over $\mathbb{F}_2$ for some element $x$ of the group  if and only if $x$ is of finite order divided by $7$.
 }}

\end{center}

\vspace{1cc}

% Introduction ---------------------------------------------------------------------------------------------------------------

\section{\bf Introduction}\label{Int}

\vspace{2cc}
%--------------------------------------------------------------------------------------------------------------------------------------------------------------------

Let $\mathbb{F}$ be a field and $G$ a group. We say that a non-zero element $\alpha$ in the group algebra $\mathbb{F}[G]$ is a zero divisor if $\alpha \beta=0$ for some non-zero $\beta \in \mathbb{F}[G]$. A famous conjecture about group algebras states that there is no zero divisor in $\mathbb{F}[G]$ whenever $G$ is  torsion-free (see \cite{K,K2}). The support of an element $\alpha=\sum_{g\in G} \alpha_g g \in \mathbb{F}[G]$, denoted by $supp(\alpha)$ is the set $\{g\in G \;|\; \alpha_g\neq 0\}$. The length of an element of $\mathbb{F}[G]$ is defined as the size of its support.   

 A recent approach to settle the conjecture is to show the non-existence of zero divisors with respect to the length of possible ones (see \cite{Ab-Ta1, AJ,pascal,S}). The case length $2$ is not so hard, however non-trivial, and has been done (e.g. \cite[Theorem 2.1]{pascal}). The first unsettled case is the existence of zero divisors of length $3$. 
 Here we study possible length $3$ zero divisors in the rational group algebras. We first note that
 \begin{thm}\label{Q}
Let $G$ be a torsion-free group and $\alpha=\alpha_1h_1+\alpha_2h_2+\alpha_3h_3$ be a zero divisor with the support size $3$ in $\mathbb{Q}[G]$. Then either $h_1+h_2+h_3$ or $h_1-h_2+h_3$  is a zero divisor in $\mathbb{Z}[G]$.
\end{thm} 
We then prove that the rational group algebras of torsion-free groups which are residually finite $p$-groups for some prime $p\neq 3$ have no zero divisor of length $3$.
   \begin{thm}\label{resi-p-group}
Let $G$ be a residually finite $p$-group for some prime number $p\not=3$. Then $\mathbb{Q}[G]$ has no zero divisor whose support size is $3$. Furthermore, if $G$ is a residually finite $3$-group then there exist no zero divisor of the form $h_1-h_2+h_3$ in $\mathbb{Q}[G]$.
\end{thm}

 We note that the determination of all zero divisors of length $3$ in  group algebras over $\mathbb{F}_2$ of cyclic groups is equivalent to find all trinomials (polynomials with 3 non-zero terms) divided by irreducible polynomials over $\mathbb{F}_2$. 
  The latter is a subject  studied in coding theory \cite{gol} and we add here some results. 
\begin{prop}\label{c-1}
Let $G$ be a group, $x\in G$, $n:=o(x)>2$.  Suppose that $\alpha\in \mathbb{F}[G]$, $1\in supp(\alpha)\subseteq \langle x\rangle$ is a zero divisor so that $\alpha\beta=0$ for some non-zero $\beta\in \mathbb{F}[G]$ where $1\in supp(\beta)$ and $|supp(\beta)|$ is minimum with respect to the property $ \alpha\beta=0$. Then, there exists at least one irreducible factor of $X^n-1$ in $\mathbb{F}[X]$ which divides $\alpha(X)$. Also, there exists at least one irreducible factor of $X^n-1$ in $\mathbb{F}[X]$ which divides $\beta(X)$.
\end{prop}

\begin{thm}\label{T-8}
Let $G$ be a group, $x\in G$, $n:=o(x)>2$, $\alpha\in \{1+x+x^{-1},1+x+x^2\}\subset \mathbb{F}_2[G]$ and $\alpha\beta=0$ for some non-zero $\beta\in \mathbb{F}_2[G]$ where $1\in supp(\beta)$ and $|supp(\beta)|$ is minimum with respect to the property $\alpha\beta=0$. Then $\beta=\beta'$ or $\beta=\beta'x^{-1}$ where $\beta':=\sum_{i=0}^{n-2}{x^{s_i}}$ such that $s_0=0$, $s_1=1$ and $s_i=s_{i-2}+3$, for all $i\in \{2,3,\dots,k\}$.
\end{thm}

\begin{thm}\label{T-10}
Let $G$ be a group, $x\in G$, $n:=o(x)>2$, $\alpha\in \{1+x+x^{-1},1+x+x^2\} \subset \mathbb{F}_2[G]$ and $\alpha\beta=0$ for some non-zero $\beta\in \mathbb{F}_2[G]$ where $1\in supp(\beta)$ and $|supp(\beta)|$ is minimum with respect to the property $\alpha\beta=0$. Then $n$ must be a multiple of $3$ and $|supp(\beta)|=2n/3$.
\end{thm}

\begin{thm}\label{T-11}
Let $G$ be a group, $x\in G$, $n:=o(x)>2$, $\alpha\in \{1+x+x^3,1+x^2+x^3\}\subset \mathbb{F}_2[G]$ be a zero divisor. Then $n$ must be a multiple of $7$.
\end{thm}
\section{\bf Zero divisors with odd support sizes}\label{S-1}
Throughout this paper, let $\mathbb{F}_q$ denote the finite field of size $q$. The following result is well-known, we mention its proof for the reader's convenience. 
\begin{prop}\label{p-group}
Let $G$ be a finite $p$-group and $\mathbb{F}$ be a field of characteristic $p$. If $\alpha:=\sum_{g\in G} \alpha_g g$ in $\mathbb{F}[G]$ is such that $\sum_{g\in G} \alpha_g \not=0$, then $\alpha$ is invertible. In particular, $\alpha$ is not a zero divisor of  $\mathbb{F}[G]$.
\end{prop}
\begin{proof}
 Let $\lambda := \sum_{g\in G} \alpha_g$. Then $\iota=\lambda \cdot 1_G -\alpha \in I(G)$, where $I(G)$ is the augmentation ideal of $\mathbb{F}[G]$.  
 Since $I(G)$ is nilpotent (see e.g. \cite[Exercise 6 (b), page 226]{robinson}), $\iota^m=0$ for some positive integer $m$ and so $(1-\lambda^{-1} \iota)\big(\sum_{i=0}^{m-1} (\lambda^{-1} \iota)^i\big)=1$. Thus $\alpha=\lambda (1-\lambda^{-1} \iota)$ is invertible.  This completes the proof. 
\end{proof}
\begin{thm}\label{resi-2-group}
Let $G$ be a residually finite $2$-group. Then $\mathbb{F}_2[G]$ has no zero divisor whose support size is odd.
\end{thm}
\begin{proof}
Suppose, for a contradiction, that $\alpha \beta=0$ for some non-zero $\alpha,\beta \in \mathbb{F}_2[G]$ such that $|supp(\alpha)|$ is odd. Let $A:=\{x^{-1}y \vert  x,y\in supp(\alpha) \text{\ or \ } x,y\in supp(\beta) \text{\ and\ } x\not=y\}$. So, there exists a normal subgroup $N$ of $G$ such that $A \cap N=\varnothing$ and $\frac{G}{N}$ is a finite $2$-group. It follows that $\overline{\alpha}\overline{\beta}=\overline{0}$, where $~^{\overline{~}}:\mathbb{F}_2[G]\rightarrow \mathbb{F}_2[\frac{G}{N}]$ is the natural ring epimorphism such that $\overline{x}=xN$ for all $x\in G$.  Since $A \cap N=\varnothing$, $|supp(\beta)|=|supp(\overline{\beta})|$ and $|supp(\alpha)|=|supp(\overline{\alpha})|$. The latter contradicts Proposition \ref{p-group}. This completes the proof. 
\end{proof}

\begin{proof}[Proof of Theorem \ref{Q}]
We may assume $\alpha$ is a zero divisor in $\mathbb{Z}[G]$ and  $\gcd (\alpha_1,\alpha_2,\alpha_3)=1$. Suppose, for a contradiction, that $p\mid\alpha_1$ but $p\nmid\alpha_2$ or $p\nmid\alpha_3$ for some prime  $p$. So $\alpha'=\alpha_2h_2+\alpha_3h_3$ is a zero divisor with the support size $1$ or $2$ in the group algebra of $G$ over the finite field of size $p$, a contradiction, since for any  field $\mathbb{F}$, $\mathbb{F}[G]$ does not contain a zero divisor whose support is of size at most $2$ (see \cite[Theorem 2.1]{pascal}). Therefore $|\alpha_1|=|\alpha_2|=|\alpha_3|$. This completes the proof.
\end{proof}
\begin{proof}[Proof of Theorem \ref{resi-p-group}]
If $p=2$ then by Theorem \ref{resi-2-group} and \cite[Theorem 2.1 and Lemma 2.2]{pascal}, the statement is obviously true. Let $G$ be a residually finite $p$-group for some prime number $p\not=2$ and $\alpha=\alpha_1h_1+\alpha_2h_2+\alpha_3h_3$ be a zero divisor with the support size $3$ in $\mathbb{Q}[G]$. Then by Theorem \ref{Q}, there exists a zero divisor of the form $h_1+h_2+h_3$ or $h_1-h_2+h_3$ in $\mathbb{Q}[G]$. So, there exists a zero divisor of the form $h_1+h_2+h_3$ or $h_1-h_2+h_3$ in $\mathbb{F}_p[G]$. Therefore, if $p>3$ then there is a contradiction by Proposition \ref{p-group} and so $\mathbb{Q}[G]$ has no zero divisor whose support size is $3$. On the other hand, if $p=3$ then by Proposition \ref{p-group}, $h_1-h_2+h_3$ can not be a zero divisor in $\mathbb{F}_p[G]$ and so there exist no zero divisor of the form $h_1-h_2+h_3$ in $\mathbb{Q}[G]$. This completes the proof. 
\end{proof}
%
%--------------------------------------------------------------------------------------------------------------------------------------------------------------------
%
\section{\bf Zero divisors of the form $1+x^i+x^j$ in group algebras over $\mathbb{F}_2$}\label{S-2}

Let $G$ be a group, $x\in G$, $n:=o(x)>2$.  Suppose that $\alpha\in \mathbb{F}_2[G]$, $1\in supp(\alpha)\subseteq \langle x\rangle$ is a zero divisor so that $\alpha\beta=0$ for some non-zero $\beta\in \mathbb{F}_2[G]$. We may assume that  $1\in supp(\beta)$ and by \cite[Lemma 2.5]{Ab-Ta1}, if we choose $\beta$ of minimum support size with respect to the property $\alpha \beta=0$, then $ supp(\beta) \subseteq \langle supp(\alpha)\rangle\subseteq \langle x\rangle$. 
  In this section, we want to study  possible values of $n$ with the following property:  There exist distinct $i,j\in\{1,2,\dots,n-1\}$ such that 
$\alpha=1+x^i+x^j$.

Let $\mathbb{F}[X]$ denote the polynomial ring in the indeterminate $X$. In the following we use the ring isomorphism $\mathbb{F}[\langle x \rangle]\cong \frac{\mathbb{F}[X]}{\langle X^n-1\rangle}$, where $\langle X^n-1\rangle$ is the ideal generated by $X^n-1$ in $\mathbb{F}[X]$. 
Actually the map $x\mapsto X$ can be extended to an epimorphism from $\mathbb{F}[\langle x \rangle]$ to $\mathbb{F}[X]$ whose kernel is $\langle X^n-1\rangle$. We denote by $\alpha(X)$ the image of $\alpha\in \mathbb{F}[\langle x \rangle]$ under the latter ring epimorphism. Note that $\alpha \beta=0$ in $\mathbb{F}[\langle x \rangle]$ is equivalent to $\alpha(X) \beta(X) \in \langle X^n-1 \rangle$.
\begin{lem}\label{irr}
Let $\alpha(X)\in \mathbb{F}[X]$ and $\alpha(X)\not\in\langle X^n-1\rangle$ for some positive integer $n$. Then, $\alpha(X)\beta(X)\in \langle X^n-1\rangle$ for some $\beta(X)\in \mathbb{F}[X]$ such that $\beta(X)\not\in\langle X^n-1\rangle$ if and only if there exists at least one irreducible factor of $X^n-1$ in $\mathbb{F}[X]$ which divides $\alpha(X)$.
\end{lem}
\begin{proof}
Suppose, for a contradiction, that $f(X)\nmid \alpha(X)$ for each irreducible polynomial $f(X)\in \mathbb{F}[X]$ which is a factor of $X^n-1$. Since $\alpha(X)\beta(X)\in \langle X^n-1\rangle$, it follows that $\beta(X)\in \langle X^n-1\rangle$,  a contradiction. 

Now suppose that $\alpha(X)=f(X) r(X)$ for some irreducible factor $f(X)$ of $X^n-1$ and some $r(X)\in \mathbb{F}[X]$. Assume that $X^n-1=f(X)\beta(X)$ for some  $\beta(X)\in \mathbb{F}[X]$. It follows that $\beta(X) \not\in \langle X^n-1\rangle$ and  $\alpha(X)\beta(X)\in \langle X^n-1\rangle$. This completes the proof.
\end{proof}

\begin{proof}[Proof of Proposition \ref{c-1}]
It follows from Lemma \ref{irr}.
\end{proof}

\begin{defn}
%primitivity=order=period
If $f\in \mathbb{F}_q[X]$ is a polynomial such that $f(0)\not=0$, then the least positive integer $t$ for which  $f$ divides $X^t-1$ is called the order of $f$.
\end{defn}
\begin{lem}\cite[Corollary 3.4]{nied}
 If $f\in \mathbb{F}_q[X]$ is an irreducible polynomial over $\mathbb{F}_q$ of degree  $m$, then the order of $f$ divides $q^m-1$.
 \end{lem}
\begin{defn}
For any field $\mathbb{F}$, a polynomial $f$ in $\mathbb{F}[X]$ is called a trinomial whenever $f$ has only three non-zero terms i.e., $f=\alpha_1 X^i+\alpha_2 X^j +\alpha_3 X^k$, where $\alpha_1,\alpha_2,\alpha_3$ are non-zero and $i,j,k$ are pairwise distinct non-negative integers.
\end{defn}  
\begin{rem}\label{r-1}
{\rm
\begin{enumerate}
\item
If $f\in \mathbb{F}_q[X]$ is an irreducible polynomial such that $f(0)\not=0$, then the order of $f$ is $t$ if and only if $f$ divides the cyclotomic polynomial $Q_t(X)=\prod_{d|t}{(1-X^{t/d})^{\mu(d)}}$, where $\mu(d)$ is the M\"obius function. Also, any monic irreducible factor of $Q_t$ has the same degree \cite{nied}. 
\item
Let $t$ be an odd positive integer and $f,g\in \mathbb{F}_2[X]$ be two distinct monic irreducible factors of $Q_t$. Then $f$ divides a trinomial if and only if $g$ divides a trinomial. Furthermore, if an irreducible polynomial in $\mathbb{F}_2[X]$ of order $t$ divides a trinomial, then any irreducible polynomial in $\mathbb{F}_2[X]$ of order $lt$ divides a trinomial, for all odd positive integers $l$ \cite{gol}.
\item
Let $G$ be a group, $x\in G$, $n:=o(x)>2$, $\alpha=1+x^i+x^j\in\mathbb{F}_2[G]$ for some distinct $i,j\in\{1,2,\dots,n-1\}$ and $\alpha\beta=0$ for some non-zero $\beta\in \mathbb{F}_2[G]$, where $1\in supp(\beta)$ and $|supp(\beta)|$ is minimum with respect to the property $ \alpha\beta=0$. Then by Proposition \ref{c-1}, there exists at least one irreducible factor of $X^n+1$ in $\mathbb{F}_2[X]$ which divides the trinomial $\alpha(X)$. If $n=2^kl$ for some odd integer $l>1$ and some positive integer $k$, then by Proposition \ref{c-1}, there exists at least one irreducible factor of $X^{l}+1$ in $\mathbb{F}_2[X]$ which divides the trinomial $\alpha(X)$, since $X^n+1=(X^l+1)^{2^k}$ in $\mathbb{F}_2[X]$. 
\end{enumerate}
}
\end{rem}
\begin{thm}\cite[Theorem 7 (Welch's Criterion)]{gol}
For any odd positive integer $t$, any irreducible polynomial in $\mathbb{F}_2[X]$ of order $t$ divides a trinomial if and only if $gcd(1+X^t,1+(1+X)^t)$ has degree greater than $1$.
\end{thm}
\begin{cor}\label{cor1}
There exists a zero divisor $\alpha$ in the group algebra $\mathbb{F}_2[\mathbb{Z}_n]$ such that $|supp(\alpha)|=3$ if and only if there exists an odd positive integer $t$ such that $t\mid n$ and $gcd(1+X^t,1+(1+X)^t)$ has degree greater than $1$ in the ring $\mathbb{F}_2[X]$.
\end{cor}
\begin{rem}
{\rm
By Corollary \ref{cor1}, if $\alpha\in \mathbb{F}_2[\mathbb{Z}_n]$, $|supp(\alpha)|=3$ and there exists an irreducible polynomial $f\in \mathbb{F}_2[X]$ of order $t$, for some odd positive integer $t\mid n$, which divides $\alpha(X)$, then $\alpha$ is a zero divisor in the group algebra $\mathbb{F}_2[\mathbb{Z}_n]$.
}
\end{rem}
\begin{exm}
Let $\alpha=x^3+x+1$ be an element of the group algebra $\mathbb{F}_2[\mathbb{Z}_7]$, where $\mathbb{Z}_7=\langle x\rangle$. It is easy to see that $\alpha(X)=X^3+X+1$ is an irreducible polynomial of order $7$ in $\mathbb{F}_2[X]$ which is already a trinomial. So, $\alpha$ is a zero divisor of $\mathbb{F}_2[\mathbb{Z}_7]$. Also,  $\alpha\beta=0$ for $\beta=(x^3+x^2+1)(x+1)\in\mathbb{F}_2[\mathbb{Z}_7]$ because $\alpha(X)\beta(X)=X^7+1$.
\end{exm}
\begin{exm}
Let $\alpha=x^{16}+x+1$ be an element of the group algebra $\mathbb{F}_2[\mathbb{Z}_{85}]$ where $\mathbb{Z}_{85}=\langle x\rangle$. It can be seen that $\alpha(X)=X^{16}+X+1=(X^8+X^6+X^5+X^3+1)(X^8+X^6+X^5+X^4+X^3+X+1)$ and $f(X)=X^8+X^6+X^5+X^4+X^3+X+1$ is an irreducible polynomial of order $85$. So, $\alpha$ is a zero divisor of $\mathbb{F}_2[\mathbb{Z}_{85}]$.
\end{exm}
\begin{rem}
{\rm
\begin{enumerate}
\item
Let $M$ be the set of all positive integers $t$ such that any irreducible polynomial in $\mathbb{F}_2[X]$ of order $t$ divides a trinomial but no irreducible polynomial of order $d$ divides a trinomial, for all $d|t$ such that $d\not=t$. It is easy to see that if $t\in M$ then $t$ is an odd positive integer greater than $1$. Also we note that over any finite field, $t\mid n$ if and only if $(X^t-1)\mid (X^n-1)$, for all positive integers $t$ and $n$. 
\item
Let $G$ be a group, $x\in G$, $n:=o(x)>2$, $\alpha=1+x^i+x^j\in\mathbb{F}[G]$, for some distinct $i,j\in\{1,2,\dots,n-1\}$, and $\alpha\beta=0$ for some non-zero $\beta\in \mathbb{F}_2[G]$ where $1\in supp(\beta)$ and $|supp(\beta)|$ is minimum with respect to the property $ \alpha\beta=0$. Then by part $(1)$ and Remark \ref{r-1}, the set of all possible values of $n$ is the set of all odd multiples of  elements of $M$.
\end{enumerate}
}
\end{rem}
In the following, some results about the elements of $M$ which are obtained in \cite{gol} are given.
\begin{thm}[See \cite{gol}]
Let $M$ be the set of all positive integers $t$ such that any irreducible polynomial in $\mathbb{F}_2[X]$ of order $t$ divides a trinomial but no irreducible polynomial of order $d$ divides a trinomial, for all $d|t$ such that $d\not=t$. Then,
\begin{enumerate}
\item
All  Mersenne primes i.e., prime numbers of the form $2^m-1$ where $m\in \mathbb{N}$, are in $M$.
\item
By computer search among non-Mersenne prime numbers smaller than $3000000$, there exist only five elements  $73,121369, 178481,262657$ and $599479$ in $M$ .
\item
By computer search among non-prime numbers smaller than $1000000$, there exist only ten elements $85,2047,3133,4369,11275,49981,60787,76627,140911$ and $486737$ in $M$.
\item
Eight other larger non-prime elements of $M$ which are currently known are 
$1826203,2304167,$ $2528921,8727391,14709241,15732721,16843009$ and $23828017$.
\end{enumerate}
\end{thm}
%
%--------------------------------------------------------------------------------------------------------------------------------------------------------------------
%
\section{\bf Zero divisors of the form $1+x^i+x^j\in\mathbb{F}_2[\langle x\rangle]$ for some special values of $(i,j)$}\label{S-3}

Let $G$ be a group, $x\in G$, $n:=o(x)>2$.  Suppose that $\alpha\in \mathbb{F}_2[G]$, $1\in supp(\alpha)\subseteq \langle x\rangle$ is a zero divisor so that $\alpha\beta=0$ for some non-zero $\beta\in \mathbb{F}_2[G]$. We may assume that  $1\in supp(\beta)$ and by \cite[Lemma 2.5]{Ab-Ta1}, if we choose $\beta$ of minimum support size with respect to the property $\alpha \beta=0$, then $ supp(\beta) \subseteq \langle supp(\alpha)\rangle\subseteq \langle x\rangle$. 

In this section, we focus on the cases that $(i,j)\in \{(1,-1),(1,2)\}\cup \{(1,3),(2,3),(2,n-1),(n-3,n-2),(1,n-2),(n-3,n-1)\}$. Firstly for the case $(i,j)\in \{(1,-1),(1,2)\}$, we show that $n$ must be a multiple of $3$ and $\beta=\beta'$ or $\beta=\beta'x^{-1}$ where $\beta':=\sum_{i=0}^{n-2}{x^{s_i}}$ such that $s_0=0$, $s_1=1$ and $s_i=s_{i-2}+3$, for all $i\in \{2,3,\dots,k\}$. Secondly for the case $(i,j)\in \{(1,3),(2,3),(2,n-1),(n-3,n-2),(1,n-2),(n-3,n-1)\}$, we show that $n$ must be a multiple of $7$. 

Throughout this section, let $A^{+}:=A+1$, $A^{2+}:=A+2$ and $A^{-}:=A-1$, where $n>1$ is a positive integer and $A$ is a subset of $\mathbb{Z}_n$.
\begin{lem}\label{T-1}
Let $A:=\{t_0,t_1,\dots,t_k\}$ be a subset of $\mathbb{Z}_n$ such that $t_0=0$, $1\leq t_1<t_2<\dots<t_k\leq n-1$ and the multiplicity of each element in $A \cup A^{+} \cup A^{2+}$ is $2$. Then for all $i\in \{1,2,\dots,k\}$, either $t_i=t_{i-1}+1$ or $t_i=t_{i-1}+2$.
\end{lem}
\begin{proof}
Since $t_0=0$ and $1\leq t_1<t_2<\dots<t_k\leq n-1$, the following inequalities are satisfied:
\begin{equation}\label{e-1}
(0=)t_0\lneqq t_1\lneqq t_1+1\leq t_2\lneqq t_2+1\leq t_3\lneqq\dots\leq t_k\leq n-1\leq t_k+1
\end{equation}
Let $i\in \{1,2,\dots,k\}$. Since the multiplicity of each element in $A \cup A^{+} \cup A^{2+}$ is $2$,  either $t_i=t_l+1(mod\;n)$  or $t_i=t_l+2(mod\;n)$, for some $l\in \{0,1,\dots,k\}$.

$(1)$ Let $t_i=t_l+1(mod\;n)$. If $i\leq l$ then $1\leq t_i\leq t_l\lneqq t_l+1\leq n$, that is a contradiction with $t_i=t_l+1(mod\;n)$. Therefore, $i>l$ and so $t_l\lneqq t_i$ and $t_l\lneqq n-1$. Since $t_i\leq n-1$ and $t_l+1\leq n-1$, we have $t_i=t_l+1$. Therefore by \ref{e-1}, $l=i-1$ and $t_i=t_{i-1}+1$.

$(2)$ Let $t_i=t_l+2(mod\;n)$. Suppose, for a contradiction, that $i\leq l$. Then $t_i\leq t_l\lneqq t_l+2\leq n+1$. If $t_l+2\lneqq n+1$, then $1\leq t_i\lneqq  t_l+2\leq n$, that is a contradiction with $t_i=t_l+2(mod\;n)$. Therefore, $t_l+2= n+1$ and so $t_i=1$, $t_l=n-1$ and $(0=)t_0=t_l+1(mod\;n)$. Hence, $t_i=t_0+1=t_l+2(mod\;n)$ i.e., the multiplicity of an element in $A \cup A^{+} \cup A^{2+}$ is greater than $2$, a contradiction. Therefore, $i>l$ and so $t_l\lneqq t_i$ and $l<k$. Since $t_l\lneqq t_i\leq n-1$, we have $t_l+2\leq n$. If $t_l+2=n$, then $t_i=0(mod\;n)$, that is a contradiction with $1\leq t_i\leq n-1$. Therefore, $t_l+2\leq n-1$ and so $t_i=t_l+2$ because $1\leq t_i\leq n-1$ and $t_i=t_l+2(mod\;n)$.

Since $l<i$ we have $l\leq i-1$. In the following we show that $l=i-1$. Suppose, for a contradiction, that $l<i-2$. By \ref{e-1}, $t_{i-3}\lneqq t_{i-3}+1\leq t_{i-2}\lneqq t_{i-2}+1\leq t_{i-1}$. So, $t_{i-3}+2\leq t_{i-1}$. On the other hand, $l\leq i-3$ and so $t_l\leq t_{i-3}$. Therefore, $t_{i-1}\lneqq t_i=t_l+2\leq t_{i-3} +2$, that is a contradiction with $t_{i-3}+2\leq t_{i-1}$. Thus $i-2\leq l\leq i-1$. Now suppose, for a contradiction, that $l=i-2$. Then $t_i=t_{i-2}+2$. By \ref{e-1}, $t_{i-2}\lneqq t_{i-2}+1\leq t_{i-1}\lneqq t_i=t_{i-2}+2$ and so $t_{i-1}=t_{i-2}+1$. Therefore, $t_{i-2}+2=t_i=t_{i-1}+1$ i.e., the multiplicity of an element in $A \cup A^{+} \cup A^{2+}$ is greater than $2$, a contradiction. So, $l=i-1$ because $i-2\lneqq l\leq i-1$. Therefore, $t_i=t_{i-1}+2$.
\end{proof}
\begin{lem}\label{T-2}
Let $A:=\{t_0,t_1,\dots,t_k\}$ be a subset of $\mathbb{Z}_n$ such that the multiplicity of each element in $A \cup A^{+} \cup A^{2+}$ is $2$. Then $t_i+1\not\in A$ or $t_i+2\not\in A$, for all $i\in \{0,1,\dots,k\}$.
\end{lem}
\begin{proof}
Suppose, for a contradiction, that there are distinct $l,s\in \{0,1,\dots,k\}$ such that $t_l=t_i+1(mod\;n)$ and $t_s=t_i+2(mod\;n)$. Then $t_s=t_l+1=t_i+2(mod\;n)$ i.e., the multiplicity of an element in $A \cup A^{+} \cup A^{2+}$ is greater than $2$, a contradiction. This completes the proof.
\end{proof}
\begin{cor}\label{T-3}
Let $A:=\{t_0,t_1,\dots,t_k\}$ be a subset of $\mathbb{Z}_n$ such that $t_0=0$, $1\leq t_1<t_2<\dots<t_k\leq n-1$ and the multiplicity of each element in $A \cup A^{+} \cup A^{2+}$ is $2$. Then either $t_i=t_{i-1}+1$ and $t_{i-1}+2\not \in A$ or $t_i=t_{i-1}+2$ and $t_{i-1}+1\not \in A$, for all $i\in \{1,2,\dots,k\}$.
\end{cor}
\begin{proof}
Lemmas \ref{T-1} and \ref{T-2} complete the proof.
\end{proof}
\begin{lem}\label{T-4}
Let $A:=\{t_0,t_1,\dots,t_k\}$ be a subset of $\mathbb{Z}_n$ such that $t_0=0$, $1\leq t_1<t_2<\dots<t_k\leq n-1$ and the multiplicity of each element in $A \cup A^{+} \cup A^{2+}$ is $2$. If $t_i-1\not\in A$ for some $i\in \{1,2,\dots,k\}$, then $t_i+1(mod\;n)\in A$. Furthermore if $i<k$, then $t_{i+1}=t_i+1$.
\end{lem}
\begin{proof}
Suppose that $t_i-1\not\in A$ for some $i\in \{1,2,\dots,k\}$. Then $t_{i-1}\lneqq t_i-1$ because $t_{i-1}+1\leq t_i$. Therefore, $t_i>t_{i-1}+1$ and so by Lemma \ref{T-1}, $t_i=t_{i-1}+2$. Since the multiplicity of each element in $A \cup A^{+} \cup A^{2+}$ is $2$, there must be an $l\in \{0,1,\dots,k\}\setminus \{i\}$ such that either $t_i+1=t_l+2(mod\;n)$ or $t_i+1=t_l(mod\;n)$. If $t_i+1=t_l+2(mod\;n)$, then $t_{i-1}+2=t_i=t_l+1(mod\;n)$ i.e., the multiplicity of an element in $A \cup A^{+} \cup A^{2+}$ is greater than $2$, a contradiction. So, $t_l=t_i+1(mod\;n)$ for some $l\in \{0,1,\dots,k\}\setminus \{i\}$. Now if $i<k$, then 
\begin{equation}\label{e-2}
(0=)t_0\lneqq t_1\lneqq t_1+1\leq \dots \leq t_i\lneqq t_i+1\leq t_{i+1}\lneqq \dots\leq t_k\leq n-1
\end{equation}
By \ref{e-2}, $t_{i+1}=t_i+1$.
\end{proof}
\begin{lem}\label{T-5}
Let $A:=\{t_0,t_1,\dots,t_k\}$ be a subset of $\mathbb{Z}_n$ such that $t_0=0$, $1\leq t_1<t_2<\dots<t_k\leq n-1$ and the multiplicity of each element in $A \cup A^{+} \cup A^{2+}$ is $2$. If $0\in A^{+}$ then $t_k=n-1$ and $1\not\in A$. Also if $0\in A^{2+}$ then $t_k=n-2$ and $1\in A$.
\end{lem}
\begin{proof}
Since the multiplicity of each element in $A \cup A^{+} \cup A^{2+}$ is $2$ and $(0=)t_0\in A$,  either $0=t_l+1(mod\;n)$  or $0=t_l+2(mod\;n)$, for some $l\in \{1,2,\dots,k\}$.

$(1)$ Let $0=t_l+1(mod\;n)$ i.e., $0\in A^{+}$. Therefore because $1\leq t_l\leq n-1$, $t_l+1=n$ and so $t_l=t_k=n-1$ and $1=t_k+2(mod\;n)$. If $1\in A$, then $t_1=1$ and $t_1=t_0+1=t_k+2(mod\;n)$ i.e., the multiplicity of an element in $A \cup A^{+} \cup A^{2+}$ is greater than $2$, a contradiction. So, $1\not\in A$.

$(2)$ Let $0=t_l+2(mod\;n)$ i.e., $0\in A^{2+}$. Therefore because $1\leq t_l\leq n-1$, we have $t_l+2=n$ and so $t_l=n-2$. So, $l=k-1$ or $l=k$ because $1\leq t_l\leq t_k\leq n-1$. If $l=k-1$ then $t_k=n-1$, $t_{k-1}=n-2$ and $t_0=t_{k-1}+2=t_k+1(mod\;n)$ i.e., the multiplicity of an element in $A \cup A^{+} \cup A^{2+}$ is greater than $2$, a contradiction. Therefore, $l=k$ and $t_k=n-2$. If $1\not\in A$, then there exist no $i\in \{1,2,\dots,k\}$ such that $t_0+1=t_i(mod\;n)$ or $t_0+1=t_i+2(mod\;n)$ because $2\leq t_i\leq n-2$ and $4\leq t_i+2\leq n$ i.e., the multiplicity of $t_0+1\in A \cup A^{+} \cup A^{2+}$ is $1$, a contradiction. So, $1\in A$.
\end{proof}
\begin{cor}\label{T-6}
Let $A:=\{t_0,t_1,\dots,t_k\}$ be a subset of $\mathbb{Z}_n$ such that $t_0=0$, $1\leq t_1<t_2<\dots<t_k\leq n-1$ and the multiplicity of each element in $A \cup A^{+} \cup A^{2+}$ is $2$. If $1\in A$ then $t_k=n-2$ and otherwise $t_k=n-1$.
\end{cor}
\begin{proof}
If $1\in A$, then by Lemma \ref{T-5}, $0\in A^{2+}$ and $t_k=n-2$. If $1\not\in A$, then by Lemma \ref{T-5}, $0\in A^{+}$ and $t_k=n-1$.
\end{proof}
\begin{thm}\label{T-7}
Let $A:=\{t_0,t_1,\dots,t_k\}$ be a subset of $\mathbb{Z}_n$ such that $t_0=0$, $1\leq t_1<t_2<\dots<t_k\leq n-1$ and the multiplicity of each element in $A \cup A^{+} \cup A^{2+}$ is $2$. Then for all $i\in \{2,3,\dots,k\}$, $t_i=t_{i-2}+3$ and one of the following cases is satisfied:
\begin{enumerate}
\item
If $1\in A$, then $t_1=1$, $t_k=n-2$ and we have
\begin{equation*}
\left\{
\begin{array}{lll}
t_i=t_{i-1}+1 & & i\text{ is odd}\\
t_i=t_{i-1}+2 & & i\text{ is even}\\
\end{array} \right.
\end{equation*}
\item
If $1\not\in A$, then $t_1=2$, $t_k=n-1$ and we have
\begin{equation*}
\left\{
\begin{array}{lll}
t_i=t_{i-1}+1 & & i\text{ is even}\\
t_i=t_{i-1}+2 & & i\text{ is odd}\\
\end{array} \right.
\end{equation*}
\end{enumerate}
\end{thm}
\begin{proof}

$(1)$ Let $1\in A$. By  $1\leq t_1<t_2<\dots<t_k\leq n-1$, we have $t_1=1$. Also by Corollary \ref{T-6}, $t_k=n-2$. We argue by induction on $i$. For $i=2$, it follows from $t_1=t_0+1$ and Lemma \ref{T-2} that $2=t_0+2=t_1+1 \not\in A$. So by Lemma \ref{T-1}, $t_2=t_1+2=3$. Now assume inductively that for $i<k$, the statement is true.
\begin{enumerate}
\item[(a)]
If $i$ is odd, then $t_i=t_{i-2}+3=t_{i-1}+1$ and so by Corollary \ref{T-3}, $t_i+1=t_{i-1}+2\not \in A$. Therefore by Lemma \ref{T-1}, $t_{i+1}=t_i+2=t_{i-1}+3$. So, $i+1$ is even and $t_{i+1}=t_{(i+1)-1}+2=t_{(i+1)-2}+3$.
\item[(b)]
If $i$ is even, then $t_i=t_{i-2}+3=t_{i-1}+2$ and so by Corollary \ref{T-3}, $t_i-1=t_{i-1}+1\not \in A$. Therefore by Lemma \ref{T-4}, $t_{i+1}=t_i+1=t_{i-1}+3$. So, $i+1$ is odd and $t_{i+1}=t_{(i+1)-1}+1=t_{(i+1)-2}+3$.
\end{enumerate} 
$(2)$ Let $1\not\in A$. By Lemma \ref{T-1}, $t_1=t_0+2=2$ because $1=t_0+1\not\in A$. Also by Corollary \ref{T-6}, $t_k=n-1$. We argue by induction on $i$. For $i=2$, it follows from Lemma \ref{T-4} that $t_2=t_1+1=3$ because $1=t_1-1\not\in A$. Now assume inductively that for $i<k$, the statement is true.
\begin{enumerate}
\item[(a)]
If $i$ is even, then $t_i=t_{i-2}+3=t_{i-1}+1$ and so by Corollary \ref{T-3}, $t_i+1=t_{i-1}+2\not \in A$. Therefore by Lemma \ref{T-1}, $t_{i+1}=t_i+2=t_{i-1}+3$. So, $i+1$ is odd and $t_{i+1}=t_{(i+1)-1}+2=t_{(i+1)-2}+3$.
\item[(b)]
If $i$ is odd, then $t_i=t_{i-2}+3=t_{i-1}+2$ and so by Corollary \ref{T-3}, $t_i-1=t_{i-1}+1\not \in A$. Therefore by Lemma \ref{T-4}, $t_{i+1}=t_i+1=t_{i-1}+3$. So, $i+1$ is even and $t_{i+1}=t_{(i+1)-1}+1=t_{(i+1)-2}+3$.
\end{enumerate} 
This completes the proof.
\end{proof}

\begin{proof}[Proof of Theorem \ref{T-8}]
Note that $n:=o(x)$ is finite because the group algebra of an infinite cyclic group has no zero-divisors (see \cite[Theorem 26.2]{pass2}). By \cite[Lemma 2.5]{Ab-Ta1}, if we choose $\beta$ of minimum support size with respect to the property $\alpha \beta=0$, then $ supp(\beta) \subseteq \langle supp(\alpha)\rangle\subseteq \langle x\rangle$. Let $\beta=\sum_{i=0}^{k}{x^{t_i}}$ such that $t_0=0$ and $1\leq t_1<t_2<\dots<t_k\leq n-1$. If $\alpha=1+x+x^{-1}$ and $\alpha'=1+x+x^2$ then $\alpha'=x\alpha$. Also $\alpha\gamma=0$ if and only if $\alpha'\gamma=0$, for some $\gamma\in \mathbb{F}_2[G]$. Let $A:=\{t_0,t_1,\dots,t_k\}$ be a subset of $\mathbb{Z}_n$. Since $(1+x+x^2)\beta=0$, the multiplicity of each element in $A \cup A^{+} \cup A^{2+}$ is $2$. So by Lemma \ref{T-7}, $t_i=t_{i-2}+3$ for all $i\in \{2,3,\dots,k\}$. Also, if $1\in A$ then $t_1=1$ and $t_k=n-2$, and if $1\not\in A$ then $t_1=2$ and $t_k=n-1$. Therefore, if $x\in supp(\beta)$ then $\beta=\beta'$ and otherwise $\beta=\beta'x^{-1}$ because $1=x^0=x^n$. This completes the proof.
\end{proof}
\begin{lem}\label{T-9}
Let $A$ be a subset of $\mathbb{Z}_n$ such that the multiplicity of each element in $A \cup A^{+} \cup A^{-}$ is at least $2$. Then $\mathbb{Z}_n=A \cup A^{+}=A^{-} \cup A=A^{+} \cup A^{-}$ and so $|A \cap A^{-}|=|A \cap A^{+}|=|A^{+} \cap A^{-}|$. Furthermore, if the multiplicity of each element in $A \cup A^{+} \cup A^{-}$ is exactly $2$, then $n$ must be a multiple of $3$, $|A|=2n/3$ and $|A \cap A^{-}|=|A \cap A^{+}|=|A^{+} \cap A^{-}|=n/3=|A|/2$. 
\end{lem}
\begin{proof}
By the hypothesis on the multiplicities, we have the following three inclusions:
$A \subseteq A^{+} \cup A^{-}$, $A^{+} \subseteq A^{-} \cup A$ and $A^{-} \subseteq A \cup A^{+}$. 

For the first part, by induction on $i$, we prove that $A+i \subseteq A \cup A^{+}$ for all $i\in \mathbb{N}$. It is clear for $i=0,1$. For $i=2$, it follows from the second inclusion that $A^{2+} \subseteq A \cup A^{+}$. Now assume inductively that $A+i \subseteq A \cup A^{+}$. It follows that 
$A+(i+1) \subseteq A^{+} \cup A^{2+} \subseteq A^{+} \cup (A \cup A^{+})= A \cup A^{+}$.
This completes the induction. Now as $\mathbb{Z}_n= \bigcup_{i=1}^n A+i$, $\mathbb{Z}_n=A \cup A^{+}$ and the above first three inclusions imply that $\mathbb{Z}_n=A \cup A^{-}=A^{+} \cup A^{-}$. Therefore $n=2|A|-|A \cap A^{-}|=2|A|-|A \cap A^{+}|=2|A|-|A^{+} \cap A^{-}|$ and so
$t:=|A \cap A^{-}|=|A \cap A^{+}|=|A^{+} \cap A^{-}|$. 

For the second part, suppose that the multiplicity of each element in $A \cup A^{+} \cup A^{-}$ is exactly $2$. Thus, $A \cap A^{+} \cap A^{-}=\varnothing$. Therefore $n=3|A|-3t=3(|A|-t)$ and so $n$ is a multiple of $3$. Also we have $n=3|A|-3t=2|A|-t$. Thus $t=|A|/2$ and so $|A|=2n/3$. This completes the proof. 
\end{proof}

\begin{proof}[Proof of Theorem \ref{T-10}]
Note that $n:=o(x)$ is finite because the group algebra of an infinite cyclic group has no zero-divisors (see \cite[Theorem 26.2]{pass2}). By \cite[Lemma 2.5]{Ab-Ta1}, if we choose $\beta$ of minimum support size with respect to the property $\alpha \beta=0$, then $ supp(\beta) \subseteq \langle supp(\alpha)\rangle\subseteq \langle x\rangle$. Let $\beta=\sum_{i=0}^{k}{x^{t_i}}$ such that $t_0=0$ and $1\leq t_1<t_2<\dots<t_k\leq n-1$. If $\alpha=1+x+x^{-1}$ and $\alpha'=1+x+x^2$ then $\alpha'=x\alpha$. Also $\alpha\gamma=0$ if and only if $\alpha'\gamma=0$, for some $\gamma\in \mathbb{F}_2[G]$. Let $A:=\{t_0,t_1,\dots,t_k\}$ be a subset of $\mathbb{Z}_n$. Since $(1+x+x^{-1})\beta=0$, the multiplicity of each element in $A \cup A^{+} \cup A^{-}$ is $2$. So by Lemma \ref{T-9}, $n$ must be a multiple of $3$. Also, $|supp(\beta)|=|A|=2n/3$.
\end{proof}
%--------------------------------------------------------------------------------------------------------------------------------------------------------------------

\begin{proof}[Proof of Theorem \ref{T-11}]
Note that $n:=o(x)$ is finite because the group algebra of an infinite cyclic group has no zero-divisors (see \cite[Theorem 26.2]{pass2}). By \cite[Lemma 2.5]{Ab-Ta1}, if we choose $\beta$ of minimum support size with respect to the property $\alpha \beta=0$, then $ supp(\beta) \subseteq \langle supp(\alpha)\rangle\subseteq \langle x\rangle$.  For all $\gamma=\sum_{k=0}^{n-1}{a_k x^k}$ in the group algebra $\mathbb{F}_2[\langle x\rangle]$, a polynomial $\gamma(x):=\sum_{k=0}^{n-1}{a_k x^k}$ from the ring $\mathbb{F}_2[x]$ can be corresponded to $\gamma$. So, $\alpha(x),\beta(x)\not\in \langle x^n-1\rangle$ and $\alpha(x)\beta(x)\in \langle x^n-1\rangle$. Let $\alpha=1+x+x^3$. So, $\alpha(x)=1+x+x^3$ is an irreducible polynomial in $\mathbb{F}_2[\langle x\rangle]$ and $\alpha(x) \vert x^n-1$, because $\alpha(x)\beta(x)\in \langle x^n-1\rangle$ and $\alpha(x),\beta(x)\not\in \langle x^n-1\rangle$. Therefore, the reciprocal polynomial of $\alpha(x)$ i.e., $\alpha^{*}(x)=x^{deg(\alpha(x))}\alpha(\frac{1}{x})=1+x^2+x^3$ divides $x^n-1$, too. Also, it is easy to see that $\alpha(x)$ and $\alpha^{*}(x)$ do not divide $x-1$. So, $x^7-1=\alpha(x)\alpha^{*}(x)(x-1)$ divides $x^n-1$. Furthermore, it is easy to see that $x
^a-1 \vert x^b-1$ if and only if $a\vert b$. Therefore, $n$ must be a multiple of $7$ because $x^7-1$ divides $x^n-1$. The same discussion for the case that $\alpha=1+x^2+x^3$ completes the proof because the reciprocal polynomial of $\alpha(x)$ is $\alpha^{*}(x)=1+x+x^3$.
\end{proof}
\begin{prop}\label{T-12}
Let $G$ be a group, $x\in G$, $n:=o(x)>2$, $\alpha\in \{1+x^2+x^{n-1},1+x^{n-3}+x^{n-2},1+x+x^{n-2},1+x^{n-3}+x^{n-1}\}\subset \mathbb{F}_2[G]$ be a zero divisor. Then $n$ must be a multiple of $7$.
\end{prop}
\begin{proof}
Note that $n:=o(x)$ is finite because the group algebra of an infinite cyclic group has no zero-divisors (see \cite[Theorem 26.2]{pass2}). By \cite[Lemma 2.5]{Ab-Ta1}, if we choose $\beta$ of minimum support size with respect to the property $\alpha \beta=0$, then $ supp(\beta) \subseteq \langle supp(\alpha)\rangle\subseteq \langle x\rangle$.  If $\alpha$ is equal to $1+x^2+x^{n-1}$, $1+x^{n-3}+x^{n-2}$, $1+x+x^{n-2}$ or $1+x^{n-3}+x^{n-1}$, then $\alpha'\beta=0$ where $\alpha'$ is equal to $x\alpha=1+x+x^3$, $x^3\alpha=1+x+x^3$, $x^2\alpha=1+x^2+x^3$ or $x^3\alpha=1+x^2+x^3$, respectively. So by Theorem \ref{T-11}, $n$ must be a multiple of $7$. This completes the proof.
\end{proof}
%
%-----------------------------------------------------------------------------------------------------------------------------------------------------
%


\begin{thebibliography}{40}
\bibitem{Ab-Ta1}  A. Abdollahi and Z. Taheri, Zero divisors and units with small supports in group algebras of torsion-free groups,  Comm. Algebra, 46 No. 2 (2018)  887-925. 
\bibitem{AJ} A. Abdollahi and F. Jafari, Zero divisor and unit elements with supports of size 4 in group algebras of torsion-free groups, to appear in Comm. Algebra.
\bibitem{gol}
S. W. Golomb and P. F. Lee, Irreducible polynomials which divide trinomials over $GF(2)$, IEEE Trans. Inf. Theory, 53 (2007), no. 2, 768-774.
\bibitem{K} I. Kaplansky, Problems in the theory of rings, in: Report of a Conference on Linear
Algebras, pp. 1-3, National Academy of Sciences-National Research Council,
publ. 502, Washington, 1957.
\bibitem{K2} I. Kaplansky, ‘Problems in the theory of rings’ revisited, Amer. Math. Monthly 77
(1970), 445–454.
\bibitem{nied}
R. Lidl and H. Niederreiter, Finite Fields, 2nd, Cambridge Univ. Press, Cambridge, 1997.
\bibitem{pascal}
P. Schweitzer, On zero divisors with small support in group rings of torsion-free groups, J. Group Theory, 16 (2013), no. 5, 667-693.
\bibitem{pass2}
D. S. Passman, Infinite group rings, Dekker, New York, 1971.
\bibitem{robinson}
D. J. S. Robinson, A course in the theory of groups, 2nd, Springer, New York, 1995.
\bibitem{S} L. J. Soelberg, Finding torsion-free groups which do not have the unique product property, M.Sc. Thesis, Brigham Young University, 2018.  \href{https://scholarsarchive.byu.edu/etd/6932}{All Theses and Dissertations. 6932.}
\end{thebibliography}
\end{document}